\documentclass[11pt]{article}

\usepackage{amssymb,amsmath,amsthm}
\usepackage{amsfonts}
\usepackage{latexsym}
\usepackage{amsfonts}
\usepackage{xcolor}
\usepackage{tikz}
\usepackage{mathtools}
\usepackage{mathrsfs}
\usepackage{hyperref}
\usepackage{float}
\usepackage{comment}

\usepackage{graphicx} 
\graphicspath{ {./images/} }

\newtheorem{theorem}{Theorem}[section] 

\newtheorem{claim}[theorem]{Claim}

\title{Chasing the Threshold Bias of the 3-AP Game}

\author{Albert Cao, Felix Christian Clemen, Sean English, Xiaojian Li,\\ Tatum Schmidt, Leeann Xoubi, Weian Yin}

\begin{document}
	\maketitle
	
	\begin{abstract}
		In a Maker-Breaker game there are two players, Maker and Breaker, where Maker wins if they create a specified structure while Breaker wins if they prevent Maker from winning indefinitely. A $3$-term arithmetic progression, or $3$-AP, is a sequence of three distinct integers $a, b, c$ such that $b-a = c-b$. The $3$-AP game is a biased Maker-Breaker game played on $[n]$ where every round Breaker selects $q$ unclaimed integers for every Maker's one integer. Maker is trying to select points such that they have a $3$-AP and Breaker is trying to prevent this. The main question of interest is determining the threshold bias $q^*(n)$, that is the minimum value of $q=q(n)$ for which Breaker has a winning strategy. Kusch, Ru\'e, Spiegel and Szab\'o initially asked this question and proved $\sqrt{n/12-1/6}\leq q^*(n)\leq \sqrt{3n}$. We find new strategies for both Maker and Breaker which improve the existing bounds to
		\[
		(1+o(1))\sqrt{\frac{n}{5.6}}  \leq q^*(n) \leq   \sqrt{2n} +O(1).
		\]
	\end{abstract}
	
	\section{Introduction}
	Positional game theory analyzes two-player games, such as the famous kid's game Tic-Tac-Toe or abstract games played on hypergraphs. Two-player games often have a framework where one player tries to build a specific structure while the other player tries to prevent this. Such games are called \emph{Maker-Breaker games}. For an excellent overview on Maker-Breaker games and more generally positional games we refer the reader to the survey~\cite{Krivelevich} by Krivelevich and the books~\cite{MR2402857} by Beck and \cite{MR4046020} by Hefetz, Krivelevich, Stojakovic and Szab\'o. Formally, a $(p,q)$-Maker-Breaker game is defined as follows.
	

	Let $H=(V,E)$ be a hypergraph. The vertex set $V$ is called the \emph{board} and the vertices will be referred to as \emph{points}. There are two players, Maker and Breaker, and on every turn, Maker first chooses $p$ points that previously have not been chosen, and then Breaker chooses $q$ points that have not been chosen before. The game ends once either Maker selects all points of an edge $e\in E$, or the entire board has been selected without this happening. Maker wins if they completely occupy all vertices of any edge $e\in E$, and Breaker wins otherwise, i.e. if they manage to claim at least one element in every set $e\in E$. 
	
	The study of Maker-Breaker games was popularized by Erd\H{o}s and Selfridge~\cite{MR327313} who found a criterion which gives a winning strategy for Breaker. Since then, a variety of Maker-Breaker games have received considerable attention in the literature. One of the most famous games is the game played on the edge set of $K_n$, where the winning sets are collections of edges corresponding to a fixed subgraph $H$. In particular, the case for $H=K_3$, known as the \emph{triangle game}, has been studied intensively \cite{Balogh,positional,trianglegame}. In this paper, we study the so-called $3$-AP game.
	
	A \emph{$k$-term arithmetic progression}, or \emph{$k$-AP}, is a set of integers of the form $a,a+d,...,a+ (k-1)d$ for some $a\in \mathbb{Z}$ and integer $d>0$. The \textit{$k$-AP game} is a $(1,q)$-Maker-Breaker game played on the first $n$ integers and the winning sets are the collection of $k$-APs, i.e.~in terms of the hypergraph definition of Maker-Breaker games, we are playing on the hypergraph $H=(V,E)$ where 
	\[
	V=[n] \quad \quad \text{and} \quad \quad
	E=\{S \subseteq V \ | \ S \text{ forms a } k\text{-AP}\}.
	\]
	
	In each turn, Maker selects one unoccupied number and then Breaker selects $q$ unoccupied numbers from $[n]$. If Maker manages to select all the numbers in a $k$-AP, then they win. Otherwise, Breaker wins. Define the \emph{threshold bias} $q^*=q^*(n)$ to be the minimum value of $q$ for which Breaker has a winning strategy on $[n]$. Kusch, Ru\'e, Spiegel and Szab\'o~\cite{cite:1} asked to find the threshold bias for the $3$-AP game and provided the following upper and lower bounds:
	\[ 
	\sqrt{\frac{n}{12}-\frac{1}{6}}\leq q^*(n) \leq \sqrt{3n}.
	\]
	
	The lower bound is deduced from a general theorem by Beck (Theorem 2 in \cite{cite:2}) providing a criterion for Maker to win. Their Breaker strategy for the upper bound is to block all $3$-AP's containing Maker’s last choice and one of their previous choices. 
	
	We improve both the upper and lower bounds by doing a more in-depth analysis of the $3$-AP game and utilizing the particular structures of $3$-APs. 
	
	\begin{theorem}
		\label{ourresult}
		The threshold bias of the $3$-AP game on $[n]$ satisfies
		\[  (1+o(1))\sqrt{\frac{n}{3+\frac{3}{2}\sqrt{3}}}  \leq q^*(n) \leq \sqrt{2n}+O(1).
		\]
	\end{theorem}
	Note that $3+\frac{3}{2}\sqrt{3} \approx 5.6$. Theorem~\ref{ourresult} reduces the multiplicative gap from $6$ to roughly $3.35$. 
	The organization of this paper is as follows. In Section~\ref{lower}, we present the proof of the lower bound of Theorem~\ref{ourresult}, and in Section~\ref{upper}, we present the proof of the upper bound of Theorem~\ref{ourresult}.
	

	\section{Lower Bound}
	\label{lower}
	In this section, we provide a strategy for Maker, yielding the lower bound on the threshold bias from Theorem~\ref{ourresult}. To motivate our strategy, we first give a definition. During the game, any unoccupied integer that forms a 3-AP with the points Maker has placed previously, we will call a \emph{threat}. Note that if at any time there are more threats on the board than Breaker has moves, Maker will win. Let 
	\begin{align*}
		q\leq (1+o(1))\sqrt{\frac{n}{3+\frac{3}{2}\sqrt{3}}} \quad \text{and} \quad x=\frac{1}{2}+\frac{1}{2\sqrt{3}}.
	\end{align*}
	The constant $3+\frac{3}{2}\sqrt{3}$ will be the maximum value of an elementary function $f(x)$ that will arise from our combinatorial argument. The maximum of such function will be reached at the corresponding value $x=\frac{1}{2}+\frac{1}{2\sqrt{3}}$.
	
	Our strategy for Maker is as follows. In the first $\lceil xq\rceil$ rounds, Maker will choose the smallest unoccupied integer of value larger than $\left\lceil\frac{n}{3}\right\rceil$. We will show that after these $\lceil xq\rceil$ rounds, Maker can choose an unoccupied point that creates over $q$ threats, and thus Maker wins the game.
	
	Let $B$ and $M$ be the set of integers occupied by Breaker and Maker, respectively, after round $\lceil xq\rceil$ was played. Let $\ell$ denote the largest integer occupied by Maker. Note that 
	\begin{equation*}
	    \ell\leq \left\lceil\frac{n}{3}\right\rceil+1 +(q+1)\left\lceil xq \right\rceil \leq  \frac{n}{3}+q^2 \leq \frac{2n}{3}
	\end{equation*}
	for sufficiently large $n$, because in each of the $\left\lceil xq \right\rceil$ rounds $q+1$ integers are played. 
	We split up the board $[n]=J_1 \cup J_2 \cup J_3 \cup J_4$ into $4$ intervals  
	\begin{align*}
		J_1=\left[1, \left\lceil \frac{n}{3} \right\rceil \right] \cap \mathbb{N}, \  J_2= \left(\left\lceil \frac{n}{3} \right\rceil,\ell \right]\cap \mathbb{N}, \ J_3=\left(\ell ,\left\lfloor \frac{2n}{3} \right\rfloor \right]\cap \mathbb{N} \ \text{and} \ J_4=\left(\left\lfloor \frac{2n}{3} \right\rfloor, n\right]\cap \mathbb{N}. 
	\end{align*}
	We set $b_i:=|J_i \cap B|$, so that $b_i$ is the number of Breaker points in the interval $J_i$. Note that $b_2+|M|=|J_2|=\ell - \lceil n/3 \rceil  $, because all of Maker's points are in $J_2$ and there is no unoccupied point in $J_2$. After the first $\lceil xq\rceil$ rounds, Breaker has played $q\lceil xq\rceil$ moves. Thus, $b_1+b_2+b_3+b_4=|B|=q\lceil xq\rceil$. 
	
	For $i\in J_3\setminus B $, we define 
	\begin{align*}
		A_i:= \left\{ 2y-i, 2i-y   \  \bigg\vert \ y\in M  \right\} 
	\end{align*}
	to be the set of potential threats created by left and right reflections. Note that $A_i\subseteq [n]$, because $i,y\in \left(\left\lceil \frac{n}{3} \right\rceil, \left\lfloor \frac{2n}{3} \right\rfloor\right]$. Every element in $A_i$ completes a $3$-AP with one of Makers initial $\lceil xq\rceil$ moves and the point $i$.
	 Note that $|A_i|= 2\lceil xq\rceil$. If $A_i\cap M \neq \emptyset$ for some value of $i$, then Maker wins the game by playing the integer $i$. Therefore, we can assume that $A_i\cap M =\emptyset$ for all $i\in J_3\setminus B $.
	
	Further, we can assume that $|A_i\setminus B|\leq q$ for each $i$, otherwise Maker can win in two turns. Indeed, if $|A_i\setminus B|> q$, Maker can occupy the integer $i$ and then Breaker will not be able to block all the integers in $A_i\setminus B$. Maker then wins the game by playing an unoccupied integer in $A_i\setminus B$. Therefore, we can conclude that $|A_i \cap B|\geq 2\lceil xq\rceil-q\geq (2x-1)q$. Summing up $|A_i\cap B|$ for all $i\in I_3\setminus B$, we get
	\begin{align}
		\label{doublecounting1}
		\sum_{i\in J_3\setminus B} |A_i \cap B| \geq \left|J_3\setminus B \right|  (2x-1)q .
	\end{align}
	We can use a double counting argument to upper bound the sum on the left side of \eqref{doublecounting1}. For every $b\in B$ the sum adds up the number of times $b$ appears in some $A_i$. This is at most $|M|=\lceil xq\rceil$ times for every $b\in B$. Thus,
	\begin{align}
		\label{doublecounting2}
		\sum_{i\in J_3 \setminus B} |A_i \cap B|&=\sum_{b\in B} \left| \left\{i\in J_3 \setminus B : b\in A_i \right\} \right| \leq |B|\lceil xq\rceil=q\lceil xq\rceil^2.
	\end{align}
	Combining \eqref{doublecounting1} with \eqref{doublecounting2}, we get 
	\begin{align}
	\label{countin3}
		\left( |J_3|-b_3\right)(2x-1)q \leq q(xq+1)^2.
	\end{align}
	Using that $\ell=\lceil n/3\rceil+b_2+|M|=\lceil n/3\rceil+b_2+O(\sqrt{n})$, Inequality \eqref{countin3} can be written as
	\begin{align*}
		\frac{n}{3}-b_2 -O(\sqrt{n})= \left\lfloor \frac{2n}{3} \right\rfloor -\ell =|J_3| \leq b_3+ \frac{x^2}{2x-1}q^2 +O(\sqrt{n}).
	\end{align*}
	Therefore, 
	\begin{align*}
		n \leq 3b_2+ 3b_3+ \frac{3x^2}{2x-1}q^2 +O(\sqrt{n}) &\leq  \left(3x+\frac{3x^2}{2x-1}\right)q^2 +O(\sqrt{n})\\
		&\leq \left(3+\frac{3}{2}\sqrt{3}\right)q^2+O(\sqrt{n}),
	\end{align*}    
	where we have used that $b_2+b_3\leq|B|\leq q\lceil xq\rceil$ and $x=\frac{1}{2}+\frac{1}{2\sqrt{3}}$, which we chose to minimize the expression on the right hand side. 
	This contradicts $q\leq (1+o(1))\sqrt{\frac{n}{3+\frac{3}{2}\sqrt{3}}}$. Thus, there exists an $i\in J_3 \setminus B$ such that $|A_i\setminus B|>q$ and therefore Maker can win in two more rounds, completing the proof of the lower bound on the threshold bias from Theorem~\ref{ourresult}.  \hfill \ensuremath{\Box}

	\section{Upper bound}
	\label{upper}
	In this section, we provide a strategy for Breaker, yielding the upper bound on the threshold bias from Theorem~\ref{ourresult}. Our Breaker strategy builds upon the breaker strategy from Kusch, Ru\'e, Spiegel and Szab\'o~\cite{cite:1}. First, we will briefly explain their strategy to help give some intuition for how our strategy works. 
	
	Their Breaker strategy is simply to block all threats created on every turn. Maker selects at most $\lceil n/(q+1) \rceil$ integers during the game and for each pair of integers there are at most three $3$-AP's containing them, i.e. three threats. Thus,
	Breaker has to block at most 
	\[ 
	3 \left(\left\lceil \frac{n}{q+1} \right\rceil-1 \right) \leq \frac{3n}{q}  
	\]
	integers in each round to prevent Maker from winning. Thus, Breaker can block Maker from winning (and consequently wins themselves) as long as $\frac{3n}{q}\leq q$, or equivalently as long as $q\geq \sqrt{3n}$.
	This proves $q^*(n)\leq \sqrt{3n}$. 
	
	Now, we will explain our strategy and show that it improves the upper bound for the threshold bias. Let 
	\begin{align}
		\label{biasbound}
		q\in \left[\sqrt{2n}+O(1),\sqrt{3n}\right].
	\end{align}
	We classify Breaker's moves into two different types: forced moves and free moves. \emph{Forced moves} are moves that Breaker must play to block threats. \emph{Free moves} are any moves that Breaker has after Breaker blocks all threats. Our improvement over the pre-existing Breaker strategy is that we strategically use our free moves to reduce the maximum number of threats that occur on any round.
	
	We partition the board $[n]=I_1 \cup I_2 \cup I_3$ into three consecutive intervals
	\begin{align*}
		I_1=\left[1, \left\lfloor\frac{n}{3}\right\rfloor \right] \cap \mathbb{N}, \quad  I_2=\left(\left\lfloor\frac{n}{3}\right\rfloor ,\left\lceil\frac{2n}{3}\right\rceil  \right]\cap \mathbb{N}, \quad  \text{and} \quad I_3=\left(\left\lceil\frac{2n}{3}\right\rceil , n \right]\cap \mathbb{N}, \quad 
	\end{align*}
	which we will refer to as \emph{left}, \emph{middle} and \emph{right intervals}, respectively. 
	In our strategy, Breaker plays all their free moves arbitrarily in the middle interval until the interval is entirely occupied. Let $t^*$ be the first round in which all integers of the middle interval have been chosen (note that it is not obvious that $t^*$ is well-defined as theoretically Maker could win prior to the middle becoming full; we will show that this does not happen). On turn $t^*$, after the middle interval is completely occupied, Breaker might have some free moves left and can play those remaining ones arbitrarily.
	
     In turns $t>t^*$, Breaker will play all its free moves arbitrarily in the interval that Maker played in on that turn. Under the assumption that Maker does not win, eventually this will lead to one of $I_1$ or $I_3$ becoming fully occupied, and on this turn, Breaker can play any remaining free moves arbitrarily.  
	
	A key observation for our strategy is that after time $t^*$ any new Maker point can create at most two (instead of three) threats with any previously played Maker point in the middle, see Figure~\ref{figurethreats} for an illustration.
\begin{claim}
Let $t>t^*$. Any new Maker point can create at most two threats with any previously played Maker point in the middle interval, $I_2$
\end{claim}
	\begin{proof}
Let $m\in [n]$ be a new Maker point played after time $t^*$ and $m'\in I_2$ be an integer occupied by Maker. First assume that $m\in I_1$. Only the three integers $(m+m')/2, 2m'-m$ and $2m-m'$ can be threats. Assume that $(m+m')/2$ and $2m'-m$ are threats. We will show that $2m-m'$ is not a threat.  Since $(m+m')/2<m'$, we have $(m+m')/2\in I_1$ to be a threat. Since $2m'-m>m'$, we have $2m'-m\in I_3$ to be a threat. Therefore,
	\begin{align*}
	    \frac{m+m'}{2}\leq \left\lfloor \frac{n}{3} \right\rfloor \quad \text{and} \quad \left\lceil \frac{2n}{3}\right\rceil <2m'-m \leq n.
	\end{align*}
Hence,
\begin{align*}
    2m-m'= (m+m') + (m-2m')< 2 \left\lfloor \frac{n}{3}\right\rfloor -\left\lceil \frac{2n}{3}\right\rceil \leq 0,
\end{align*}
and therefore $2m-m'$ cannot be a threat. We conclude that $m'$ together with $m$ creates at most two threats. If $m\in I_3$, the analysis is similar. 
\end{proof}
	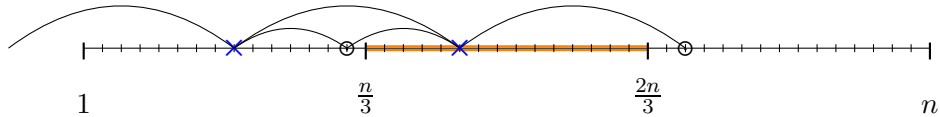
\begin{figure}[htp]
		\centering
		\begin{tikzpicture}[scale=.5]
			\draw (0,0) -- (22.5,0) ; 
			\draw[orange, thick] (7.5, 0.05) -- (15,0.05);
			\draw[orange, thick] (7.5, -0.05) -- (15,-0.05);
			\foreach \x in  {0, 0.5, 1, 1.5, 2, 2.5, 3, 3.5, 4, 4.5, 5, 5.5, 6, 6.5, 7, 7.5, 8, 8.5, 9, 9.5, 10, 10.5, 11, 11.5, 12, 12.5, 13, 13.5, 14, 14.5, 15, 15.5, 16, 16.5, 17, 17.5, 18, 18.5, 19, 19.5, 20, 20.5, 21, 21.5, 22, 22.5} 
			\draw[shift={(\x,0)},color=black] (0pt,3pt) -- (0pt,-3pt);
			
			\draw[white] (7.5, 0.2) |- (7.5, - 2) node[above] 
			{\textcolor{black}{$\frac{n}{3}$}} |- (15, -2) node[above] {\textcolor{black}{$\frac{2n}{3}$}} -|(15, 0.2);
			
			\draw[white] (0, 0.2) |- (0, - 2) node[above] 
			{\textcolor{black}{$1$}} |- (22.5, -2) node[above] {\textcolor{black}{$n$}} -|(22.5, 0.2);
			
			\draw[thick, black] (0,0.15) -- (0,-0.3);
			\draw[thick, black] (7.5,0.15) -- (7.5,-0.3);
			\draw[thick, black] (15,0.15) -- (15,-0.3);
			\draw[thick, black] (22.5,0.15) -- (22.5,-0.3);
			
			\draw (10,0) circle (0.5pt) node[] {\textcolor{blue}{$\boldsymbol \times$}};
			
			\draw (4,0) circle (0.5pt) node[] {\textcolor{blue}{$\boldsymbol \times$}};
			
			\draw (-2, 0) .. controls (0, 1.5) and (2, 1.5) .. node {} (4, 0);
			
			\draw (4, 0) .. controls (6, 1.5) and (8, 1.5) .. node {} (10, 0);
			
			\draw (10, 0) .. controls (12, 1.5) and (14, 1.5) .. node {} (16, 0);
			
			\draw (4, 0) .. controls (5, 0.7) and (6, 0.7) .. node {} (7, 0);
			
			\draw (7, 0) .. controls (8, 0.7) and (9, 0.7) .. node {} (10, 0);
			
			\draw[] (7,-0.03) circle (0.5pt) node[] {\Large $\circ$};
			
			\draw[] (16,-0.03) circle (0.5pt) node[] {\Large $\circ$};
			
		\end{tikzpicture}
		\label{figurethreats}
		\caption{Pairs of Maker points with one being in the middle interval create at most two threats. $\times$'s represent Maker points, and $\circ$'s represents the threats created by those points.}
	\end{figure}
Another important observation is that any new Maker point cannot create any threats with Maker points on the other side of the middle. 
\begin{claim}
Let $t>t^*$. A new Maker point cannot create a threat with a Maker point on the other side of the middle.
\end{claim}
\begin{proof}
This claim simply follows from the fact that for $m\in I_1$ and $m'\in I_3$, we have $(m+m')/2,2m'-m,2m-m'\notin I_1\cup I_3$, because 
\begin{gather*}
    \frac{m+m'}{2}\leq \frac{\left\lfloor \frac{n}{3}\right\rfloor+n}{2}\leq \left\lceil\frac{2n}{3} \right\rceil, \quad     \frac{m+m'}{2}> \frac{1+\left\lceil \frac{2n}{3}\right\rceil}{2}\geq \left\lfloor\frac{n}{3} \right\rfloor, \\
   2m'-m> 2\left\lceil \frac{2n}{3}\right\rceil - \left\lfloor \frac{n}{3}\right\rfloor \geq n \quad \text{and} \quad 2m-m'< 2\left\lfloor \frac{n}{3}\right\rfloor -2\left\lceil \frac{2n}{3}\right\rceil\leq 0. 
\end{gather*}
\end{proof}
Now, let us get an upper bound on when turn $t^*$ occurs.
\begin{claim}
\label{claimt^*}
    We have $t^*\leq \frac{q}{3}$.
\end{claim}
\begin{proof}
On any turn $t$ there is a maximum of $3(t-1)$ forced moves so Breaker has a minimum of $q - 3(t-1)$ free moves on this turn. Thus, for any turn $t$ prior to the middle interval being filled, Breaker will have played a total of at least
	\begin{align}
		\label{middlefull}
		\sum ^{t}_{j=1} q - 3(j-1) = tq - \frac{3}{2}(t^2-t) \geq t q- \frac{3}{2}t^2 
	\end{align} 
	free moves, all of which Breaker played in the middle interval. Since $|I_2|\leq \frac{n}{3}+2$, once the total free moves Breaker has been able to play exceeds $\frac{n}{3}+2$, the middle third is guaranteed to be completely occupied. Thus, using~\eqref{middlefull}, we conclude that
	\begin{equation*}
		t^* \leq \frac{q}{3}-\frac{1}{3}\sqrt{q^2-2n-12}\leq \frac{q}{3}.
			\qedhere
	\end{equation*}
\end{proof}

	Note that since $q-3t^*\geq0$, Breaker will always have free moves before time $t^*$ and in particular, Breaker does not lose the game before time $t^*$.
	
	Given any time $t$, let $m_1(t)$, $m_2(t)$ and $m_3(t)$ denote the number of Maker points in the left, middle and right interval respectively, and let $b_1(t)$, $b_2(t)$ and $b_3(t)$ denote the number of Breaker points in the left, middle and right interval, respectively, after turn $t$ was completed.
	
	Assume it is Breaker's turn in some round $t> t^*$. We will show that Breaker can occupy all threats in this round, which further implies that Breaker eventually wins the game. Without loss of generality, Maker has played their last move in say $I_1$ (if they play in $I_3$, the analysis is similar). We claim that Maker has created at most
	\[
	3m_1(t)+2m_2(t)=3m_1(t)+2m_2(t^*)
	\]
	threats. Indeed, there are at most $3m_1(t)$ threats that were made by the new Maker point with old Maker points in the left interval, and then there are at most $2m_2(t)=2m_2(t^*)$ threats made by the new point with old Maker points from the middle interval. There are no threats made with old Maker points from the right interval, as all integers completing a $3$-AP are either outside of the board or in the middle interval and therefore blocked already. 
	Thus, as long as 
	\begin{equation}\label{equation main inequality upper bound}
		3m_1(t)+2m_2(t^*)\leq q,
	\end{equation}
	Breaker will block all threats created by Maker in round $t$. 
	
	In the remaining part of the proof, we show that \eqref{equation main inequality upper bound} holds. To do so we collect a few equalities and inequalities which will help us to upper bound $m_1(t)$, the first term in \eqref{equation main inequality upper bound}. Note that for all $t$,
	\begin{align*}
		m_1(t)+m_2(t)+m_3(t)=t \quad \text{ and } \quad b_1(t)+b_2(t)+b_3(t)=qt 
	\end{align*}
	because Maker occupies one integer and Breaker $q$ integers in each turn. 
	Breaker plays at most
	\begin{align}
		\label{beforetimet*}
		b_3(t^*)\leq 3\binom{m_3(t^*)}{2}+2m_2(t^*)m_3(t^*)+\binom{m_2(t^*)}{2}+m_1(t^*)m_2(t^*)+O(\sqrt{n})
	\end{align}
	points in the last third until round $t^*$, because a pair of Maker points $b,b'$ creates at most
	
	\vspace{0.4cm}
	\begin{minipage}[t]{.5\textwidth}
		\begin{itemize}
			\item 3 threats in $I_3$ if $b,b'\in I_3$,
			\item 2 threats in $I_3$ if $b\in I_3, b'\in I_2$,
			\item 1 threat in $I_3$ if $b,b'\in I_2$,
		\end{itemize}   
	\end{minipage}
	\begin{minipage}[t]{.5\textwidth}
		\begin{itemize}
			\item 1 threat in $I_3$ if $b\in I_2, b'\in I_1$,
			\item 0 threats in $I_3$ if $b,b'\in I_1$,
			\item 0 threats in $I_3$ if $b\in I_1, b'\in I_3$.
		\end{itemize}
	\end{minipage}
	\vspace{0.4cm}
	
	In round $t^*$, Breaker played at most $q=O(\sqrt{n})$ integers in $I_3$, explaining the $O(\sqrt{n})$-term in \eqref{beforetimet*}. 
	In each of the turns between time $t^*$ and $t$ where Maker plays in the first third (there are $m_1(t)-m_1(t^*)$ many of those turns) Breaker plays at most $m_2(t^*)$ moves in $I_3$. Those are the forced moves which come from threats created by a new Maker point in the first third and Maker points in the middle. In each of the turns between time $t^*$ and $t$ where Maker plays in the last third, Breaker trivially plays at most $q$ moves in $I_3$. Thus, between time $t^*$ and $t$, Breaker plays at most 
	\begin{align}
		\label{aftertimet*}
		b_3(t)-b_3(t^*)\leq  q\left(m_3(t)-m_3(t^*)\right)+m_2(t^*)\left(m_1(t)-m_1(t^*)\right)+O(\sqrt{n})
	\end{align}
	integers in the last third. Note that 
	\begin{align}
		\label{breakerbound}
		\left\lceil \frac{2n}{3} \right\rceil= |I_1\cup I_2|\geq b_1(t)+b_2(t)=qt-b_3(t)=q(m_1(t)+m_2(t)+m_3(t))-b_3(t).
	\end{align}
	Combining \eqref{beforetimet*}, \eqref{aftertimet*} and \eqref{breakerbound}, we get that 
	\begin{align}
		\label{inequalitym1}
		q(&m_1(t)+m_2(t)+m_3(t))-3\binom{m_3(t^*)}{2}-2m_2(t^*)m_3(t^*) -\binom{m_2(t^*)}{2}-m_1(t^*)m_2(t^*)  \nonumber \\   -&q(m_3(t)-m_3(t^*))-m_2(t^*)(m_1(t)-m_1(t^*))
		\leq \frac{2n}{3}+O(\sqrt{n}).
	\end{align}
	Inequality \eqref{inequalitym1} can be simplified to 
	\begin{align}
		\label{ineqonm1}
		m_1&(t)(q-m_2(t^*))+q(m_2(t^*)+m_3(t))-\frac{3}{2}m_3(t^*)^2-2m_2(t^*)m_3(t^*)-\frac{1}{2}m_2(t^*)^2\nonumber \\
		-&q(m_3(t)-m_3(t^*)) \leq \frac{2n}{3} +O(\sqrt{n}).
	\end{align}
	Rewriting \eqref{ineqonm1}, $m_1(t)$ can be upper bounded by 
	\begin{align}
		\label{boundonm1}
		\nonumber
		& \ \ \ m_1(t)\leq  \frac{1}{(q-m_2(t^*))} \bigg(\frac{2n}{3}-q(m_2(t^*)+m_3(t))+\frac{3}{2}m_3(t^*)^2+2m_2(t^*)m_3(t^*)+\frac{1}{2}m_2(t^*)^2\\
		&+q(m_3(t)-m_3(t^*))\bigg) +O(1) \nonumber\\
		&=\frac{1}{(q-m_2(t^*))} \bigg(\frac{2n}{3}-qm_2(t^*)+\frac{1}{2}m_2(t^*)^2 +m_3(t^*)\left(\frac{3}{2}m_3(t^*)+2m_2(t^*)-q\right)\bigg)+O(1) \nonumber\\
		&\leq \frac{1}{(q-m_2(t^*))} \bigg(\frac{2n}{3}-qm_2(t^*)+\frac{1}{2}m_2(t^*)^2\bigg)+O(1),
	\end{align}
	where we used Claim~\ref{claimt^*} to get $\frac{3}{2}m_3(t^*)+2m_2(t^*)\leq 2t^*\leq \frac{2q}{3}$ in the last inequality. Finally, we upper bound the left hand side of \eqref{equation main inequality upper bound} using our bound on $m_1(t)$ from \eqref{boundonm1},
	\begin{align}
		\label{upperboundofthreats}
		\nonumber
		3m_1(t)+2m_2(t^*)&\leq \frac{3}{(q-m_2(t^*))} \bigg(\frac{2n}{3}-qm_2(t^*)+\frac{1}{2}m_2(t^*)^2\bigg)+2m_2(t^*)+O(1)\\
		&\leq \frac{3}{(\sqrt{2n}-m_2(t^*))} \bigg(\frac{2n}{3}-\sqrt{2n}m_2(t^*)+\frac{1}{2}m_2(t^*)^2\bigg)+2m_2(t^*)+O(1).
	\end{align}
	Let $c\geq0$ be such that $m_2(t^*)=c\sqrt{n}$. Since $m_2(t^*)\leq t^*\leq q/3$, we have $c < 1$. Define 
	\[
	f(x):=\frac{3}{\sqrt{2}-x}\left(\frac{2}{3}-\sqrt{2}x+\frac{x^2}{2}\right)+2x,
	\]
	and note that the right-hand side of \eqref{upperboundofthreats} is $f(c)\sqrt{n}+O(1)$. The first derivative of $f$ is non-positive for $0\leq x\leq 1$,
	\begin{align*}
		\frac{d f}{d x}=\frac{c(c-2\sqrt{2})}{2(\sqrt{2}-c)^2}\leq 0.
	\end{align*}
	Therefore, $f(c)\leq f(0)=\sqrt{2}$. We conclude 
	\begin{align*}
		3m_1(t)+2m_2(t^*)\leq f(c)\sqrt{n}+O(1)\leq \sqrt{2n}+O(1)\leq q.
	\end{align*}
	Hence, inequality \eqref{equation main inequality upper bound} holds, completing the proof of the upper bound on the threshold bias from Theorem~\ref{ourresult}.\hfill \ensuremath{\Box}

	\section{Concluding remarks}
	We would like to remark that both strategies presented in this paper focus on first occupying integers in the middle. This matches the intuition that those integers are the most valuable as they are contained in the most $3$-AP's. We believe that optimal strategies will be in this spirit too.  
	
	A natural variant of the 3-AP Game on $[n]$ is the 3-AP Game on $\mathbb{Z}/n\mathbb{Z}$, i.e. the winning sets are triples forming a 3-AP in the cyclic group $\mathbb{Z}/n\mathbb{Z}$. Denote the threshold bias of this game $q^*_c(n)$. We have \begin{equation}
	\label{3APZnz}
	   (1+o(1))\sqrt{\frac{n}{3+\frac{3}{2}\sqrt{3}}}  \leq q^*(n)\leq q^*_c(n)\leq \sqrt{3n}.
	\end{equation}
Note that $q^*(n)\leq q^*_c(n)$ holds trivially, since the winning sets of the 3-AP game on $[n]$ can be viewed as subsets of the winning sets of the 3-AP game on $\mathbb{Z}/n\mathbb{Z}$ and therefore Maker will win on $\mathbb{Z}/n\mathbb{Z}$ if they play with their winning strategy on $[n]$.

The upper bound in \eqref{3APZnz} follows from the same proof used in~\cite{cite:1} (see the first few paragraphs of Section~\ref{upper} for a summary of this strategy), as the strategy only relied on the fact that every pair of integers is in at most three $3$-AP's, which is true both in $[n]$ and $\mathbb{Z}/n\mathbb{Z}$. The upper bound proof from Section~\ref{upper} does not translate to this setting, since our improvement over the bound in~\cite{cite:1} comes specifically from exploiting the fact that integers in $[n]$ near $1$ or $n$ are in less $3$-AP's than the integers in the middle of $[n]$ (which is not true for $\mathbb{Z}/n\mathbb{Z}$).
	
	Another direction of further study could be to take a look at the more general $k$-AP game. The right order of magnitude for the threshold bias was established to be $n^{1/(k-1)}$ for any $k$ in \cite{cite:1}, however they do not give explicit constants. Note that 
	\begin{align*}
	    q_k^*(n)\geq \left(\frac{n}{k(k-1)^2}\right)^{1/(k-1)}
	\end{align*}
follows from a theorem by Beck (Theorem 2 in \cite{cite:2}). 
	
	Another interesting variation of the 3-AP Game is the \emph{Schur game}. This is the Maker-Breaker game on $[n]$ with the winning sets being \emph{Schur triples}, i.e. three distinct integers $a,b,c$ such that $a+b=c$. Denote $q^*_s(n)$ the threshold bias of the Schur game. Since any pair of integers can at most be in two Schur triples, the bounds
	\begin{align}
	\label{Schurgame}
	    (1+o(1))\sqrt{\frac{n}{8}}\leq q^*_s(n)\leq \sqrt{2n}
	\end{align}
	can be established quickly. The lower bound in \eqref{Schurgame} follows from Beck's theorem (Theorem 2 in \cite{cite:2}) and the upper bound follows from the same proof used in~\cite{cite:1} for the 3-AP game (see the first few paragraphs of Section~\ref{upper} for a summary of this strategy). The proof techniques we used for Theorem~\ref{ourresult} can be applied to get improvements on the bounds \eqref{Schurgame} for the Schur game. 
	
	\section{Acknowledgments}
	The authors would like to thank the Illinois Geometry Lab for facilitating this research project, and also would like to thank the anonymous referees, whose comments helped improve the paper. The second author would like to thank J\'ozsef Balogh for introducing this problem to him. This material is based upon work supported by the National Science Foundation under Grant No. DMS-1449269. Any opinions, findings, and conclusions or recommendations expressed in this material are those of the authors and do not necessarily reflect the views of the National Science Foundation.

	\bibliographystyle{abbrv}
	\bibliography{3APbib}

\end{document}